\documentclass[a4paper,fleqn]{article}

\usepackage{authblk}

\usepackage{amsthm,amssymb,amsbsy,amsmath,amsfonts,amssymb,amscd,mathrsfs}
\usepackage{bbold}
\usepackage{enumitem}
\usepackage{graphicx}
\usepackage{color}

\newtheorem{prop}{Proposition}[section]
\newtheorem{coro}{Corollary}[section]
\newtheorem{lem}{Lemma}[section]
\newtheorem{rem}{Remark}[section]
\newtheorem{assum}{Assumption}[]
\newtheorem{assums}[assum]{Assumptions}
\newtheorem{hypos}[assum]{Hypotheses}
\newtheorem{definition}{Definition}
\newtheorem{example}{Example}

\newcommand{\Rset}{\mathbb{R}}

\newcommand{\ds}{\displaystyle}

\DeclareMathOperator{\proj}{proj}

\title{Optimal synthesis for a class of $L^\infty$ optimal control problems in the plane with $L^1$ constraint on the input}
\author[1]{Emilio Molina}
\author[2]{Alain Rapaport}
\affil[1]{GIPSA-lab, Univ. Grenoble Alpes, CNRS, Grenoble INP, Grenoble,  France}
\affil[ ]{\texttt{emilio.molina-olivares@gipsa-lab.fr}}
\affil[2]{UMR MISTEA, Univ. Montpellier, INRAE, Institut Agro, Montpellier, France}
\affil[ ]{\texttt{alain.rapaport@inrae.fr}}
\date{\today}

\begin{document}
	
	\maketitle
	
	\begin{abstract}
		For a particular class of planar dynamics that are linear with respect to the control variable, we show that the feedback strategy "null-singular-null" is minimizing the maximum of a coordinate over infinite horizon, under a $L^1$ budget constraint on the control. Moreover, we characterize the optimal cost as a function of the budget. The proof is based on an unusual use of the clock form. This result generalizes the one obtained formerly for the SIR epidemiological model to more general Kolmogorov dynamics, that we illustrate on other biological models.\\
		
		
		
	\end{abstract}

	\section{Introduction}
	
	The synthesis of optimal solutions of control problems with maximum cost has received relatively few attention in the literature, apart characterizations of the value function in terms of Hamilton-Jacobi-Bellman variational inequality \cite{BarronIshii,GonzalezAragone}. For concrete problems, it is often very difficult or merely impossible to find analytical solutions, but these characterizations has led to dedicated numerical schemes \cite{DiMarcoGonzalez1999,GianattiAragoneLolitoParente}. On the other hand, necessary optimality conditions cannot be provided by a direct application of the Pontryagin's Maximum Principle, because the maximum cost is not a criterion in Mayer or Bolza form. However, equivalent formulations in Mayer form can be obtained by augmenting the state dynamics and adding a state constraint \cite{JOTA}. In practice, dealing with the state constraint remains an issue to derive analytical optimal strategies. Recently, the optimal control problem of minimizing the peak of an epidemic has been investigated for the well-known epidemiological SIR model \cite{KermackMcKendrick}. The authors have provided the optimal solution when controlling the transmission rate under a budget or $L^1$ constraint on the control \cite{Automatica}. This problem, whose trajectories lie in the plane, presents a singular arc and has been solved analytically by applying a clock form. The clock form technique is well-known for planar optimal control problem with integral or minimal time criterion \cite{Miele,HermesLaSalle}, as a tool to compare a candidate optimal solution with any other admissible solutions. Therefore, it cannot be applied in this way for  comparing maximum costs. However, for this particular epidemiological problem, the authors have used the clock form in reasoning by the absurd, showing that a possible better solution would require a larger budget. While the proof has been derived for the particular dynamics of the SIR model, the aim of the present work is to extend this technique to more generic problems of minimizing the peak of one coordinate of planar dynamics under a  $L^1$ constraint on the control. We characterize a class of problems for which the optimal solution presents the same structure of the control strategy, which consists in three phases: 1. go as fast as possible to the optimal peak value 2. apply a control to maintain the peak value constant until the budget is exhausted (singular arc),  and 3. release the control when entering a domain of the state space for which the peak cannot increasing applying the null control. Moreover, we give conditions on particular Kolmogorov dynamics in plane, for which this result generalizes the one obtained previously for the SIR epidemiological model, which can be then obtained as a simple application of our result.\\

	\medskip
	
	More precisely, we consider planar controlled dynamics defined on an positively invariant domain ${\cal D}$ of $\Rset^2$
	\begin{equation}
		\label{sys}
		\left\{\begin{array}{lll}
			\dot x & = & f_1(x,y) + g_1(x,y)u\\
			\dot y & = & f_2(x,y) + g_2(x,y) u
		\end{array}
		\right.  %\qquad u \in [0,1]
	\end{equation}
	where the maps $f_1$, $f_2$, $g_1$, $g_2$ are assumed to be smooth (at least $C^1$). Given a positive number $K$ and an initial condition $(x_0,y_0)\in \mathcal{D}$, we consider the optimal control problem over infinite horizon
	\begin{equation}\label{P}
		\inf_{u(\cdot) \in {\cal U}}\sup_{t \geq 0} y(t),
	\end{equation}
	where ${\cal U}$ denotes the set of measurable functions $u(\cdot)$ that takes values in $U:=[0,1]$
	subject to the $L^1$ constraint
	\begin{equation}
		\label{constraint}
		\int_0^{+\infty} u(t)dt \leq K,
	\end{equation}
	The problem consists then in minimizing the "peak" of the variable $y$ under a  "budget" constraint on the control $u(\cdot)$. We shall say that a solution of \eqref{sys} is {\em admissible} if the control $u(\cdot)$ satisfies the constraint \eqref{constraint}.
	
	\medskip
	
		The paper is structured as follows. In the next section, we give assumptions and some preliminary results, which ensure the well-posedness of problem \eqref{P} under infinite horizon and the control strategy that we consider later. Section \ref{secNSN} defines the control strategy that what we propose to name "NSN", and gives a characterization of it. Section \ref{secmain} proves our main result about the optimality of the NSN strategy, under a flux-like condition. Finally, we consider in Section \ref{secparticular} a class of controlled Kolmogorov dynamics for which the former assumptions are satisfied. As examples, we show how our result applies straightforwardly to the SIR model and to more sophisticated biological models. 
		
%%%%%%%%%%%%%%%%%%%%%%%%%%%%%%%%%%%%%%%%%%%%%%%%%%%%%%%%%%%%%%%%%
	
	\section{Assumptions and preliminaries}
	\label{secass}
	 
We  first make the following assumptions that will be used to deal with infinite horizon, where $\proj_y$ denotes the second projection in the $(x,y)$ coordinates, 
\begin{assums} 
	\label{ass1}
	One has
	 	\begin{enumerate}[label=\roman*.]
	 		
	 	\item For any initial condition in ${\cal D}$, the solutions of the uncontrolled dynamics (that is with $u=0$) is bounded, and any other admissible solution with a lower cost $\sup_t y(t)$ are also bounded.

	 	\item The strict sub and super level sets of the function $f_2$
	 	\[
	 	{\cal D}_- := \{ (x,y) \in {\cal D} \; ; \; f_2(x,y) < 0 \},
	 	\quad
	 	{\cal D}_+ := \{ (x,y) \in {\cal D} \; ; \; f_2(x,y) > 0 \}
	 	\]
	 	are non empty.
	 	\item For any $y \in \proj_y({\cal D}_+)$, there exists a unique  $x$ such that $f_2(x,y)=0$ with $(x,y)$ in ${\cal D}$.
	 	\end{enumerate}
 	
\end{assums}

\medskip

Note that the optimal control problem \eqref{P} over infinite horizon  is well posed (i.e.~is finite) under Assumption \ref{ass1}.i.
Under these assumptions, we define the level set
\[
 {\cal D}_0 := \{ (x,y) \in {\cal D} \; ; \; f_2(x,y) = 0 \}
\]
and the function
\begin{equation}
	\label{def_xh}
x_h(y):=\left\{ x  \; ; \; (x,y) \in {\cal D}_0 \right\}, \quad y \in \proj_y({\cal D}_+)
\end{equation}
We shall also require a certain behavior of the vector fields $f$ and $g$ on the sets ${\cal D}_+$ and ${\cal D}_0$.

\begin{assums} 
	\label{ass2}
	In ${\cal D}_+$, one has the properties
		\begin{enumerate}[label=\roman*.]
		\item $f_1$ is negative and decreasing w.r.t.~$x$ and $y$
		\item $g_1$ is increasing w.r.t.~$x$ and $y$ and $f_1+g_1$ is non positive
		\item $f_2$ is increasing w.r.t.~$x$
		\item $g_2$ is decreasing w.r.t.~$x$ and $y$ and $f_2+g_2$ is negative
        \end{enumerate}
    and moreover at ${\cal D}_0$
    \begin{enumerate}[label=\roman*.]
    	\addtocounter{enumi}{4}
    	\item[v.] $f_1$ is negative
    	\item[vi.] $f_2$ is increasing w.r.t.~$x$ and non increasing w.r.t.~$y$
    \end{enumerate}
    
\end{assums}

\medskip

Then, the following Lemma gives properties of the trajectories in ${\cal D}_+$ and its complementary in ${\cal D}$, related to the infinite horizon.

\begin{lem} 
	\label{lemD+}
	Under Assumptions \ref{ass1} and \ref{ass2},
	\begin{enumerate}
		\item With the control $u=0$, the domain ${\cal D}\setminus {\cal D}_+$ is positively invariant, and from any initial condition $(x_0,y_0)$ in ${\cal D}_+$, the solution of \eqref{sys} reaches ${\cal D}_0$ in finite time.
		\item For any initial condition $(x_0,y_0)$ in ${\cal D}_+$ and optimal control $u(\cdot)$ which satisfies the constraint \eqref{constraint}, the solution of \eqref{sys} reaches (possibly in infinite time) the domain ${\cal D}_0$.
	\end{enumerate}
\end{lem}

\begin{proof}
	1. At $(x,y) \in {\cal D}_0$, one has for $u=0$
	\[
	\frac{d}{dt} f_2(x,y)= \frac{\partial f_2(x,y)}{\partial x}f_1(x,y)
	\]
	where $\frac{\partial f_2}{\partial x}f_1$ is a negative function on ${\cal D}_+$ from Assumptions \ref{ass2}.i and \ref{ass2}.iii. By continuity, one has  $\frac{\partial f_2}{\partial x}f_1\leq 0$ on ${\cal D}_0$ and we deduce the that the sub level set $\{ (x,y) \in {\cal D} \; ; \; f_2(x,y) \leq 0\}={\cal D}\setminus {\cal D}_+$ is positively invariant.
	
	With the control $u$, as long as the solution $(x(t),y(t))$ remains in ${\cal D}_+$, one has $\dot y=f_2(x,y)>0$. Therefore, one has $y(t)\geq y_0$. Moreover, as $f_2$ is increasing w.r.t.~$x$ on ${\cal D}_+ \cup {\cal D}_0$ (Assumptions \ref{ass2}.iii and \ref{ass2}.vi), one has
	\[
	f_2(x(t),y(t))>0=f_2(x_h(t),y(t)) \Rightarrow x(t)>x_h(y(t))
	\]
	Then, $f_1$ being decreasing w.r.t.~$x$ and $y$ on ${\cal D}_+$ (Assumption \ref{ass2}.i), one has also
	\[
	\dot x(t)=f_1(x(t),y(t)) \leq f_1(x_h(y(t)),y_0)
	\]
	On the other hand, one has on ${\cal D}_0$
	\[
	f_2(x_h(y),y)=0 \Rightarrow x_h^\prime(y)=-\frac{\partial_y f_2(x,y)}{\partial_x f_2(x,y)}\geq 0
	\]
	from Assumption \ref{ass2}.vi. The map $x_h$ is thus not decreasing and one has then $x_h(y(t))\geq x_h(y_0)$, from which one gets
	\[
	\dot x_1(t) \leq f_1(x_h(y_0),y_0)
	\]
	as long as $(x(t),y(t))$ remains in ${\cal D}_+$. From Assumption \ref{ass2}.v, one has also $f_1(x_h(y_0),y_0) < 0$. If  $(x(t),y(t))$ belongs to ${\cal D}_+$ for any $t\geq 0$, the solution $x(\cdot)$ is unbounded, which contradicts Assumption \ref{ass1}.1. We deduce that $(x(t),y(t))$ has to escape ${\cal D}_+$ in finite time, and consequently $y(\cdot)$ reaches its maximum $y_{max}^0$ at finite time.
	
	\medskip
	
	2. Let $(x(\cdot),y(\cdot))$ be an optimal solution, and $\bar y := \sup_t y(t) \leq y_{max}^0 < +\infty$.
	Note that $x(\cdot)$ is not increasing in the domain ${\cal D}_+$,  with Assumption \ref{ass2}.ii. Then, with Assumption \ref{ass2}.iv, one has
	\[
	\dot y(t) \geq f_2(x(t),y(t))+ g_2(x_0,\bar y)u(t)
	\]
	as long as the solution $(x(t),y(t))$ remains in ${\cal D}_+$. Therefore, one has
	\[
	\bar y \geq y(t) \geq y_0 + \int_0^t f_2(x(t),y(t))dt+g_2(x_0,\bar y)\int_0^t u(t)dt
	\]
	and as $g_2 < -f_2<0$ is negative, one obtains
	\[
	\int_0^t f_2(x(t),y(t))dt < \bar y - y_0 - g_2(x_0,\bar y)K < +\infty
	\]
	If the solution $(x(t),y(t))$ remains in ${\cal D}_+$ for any $t \geq 0$, then $f_2(x(t),y(t))$ tends to $0$ when $t$ tends to $+_\infty$, the function $f_2$ being positive in ${\cal D_+}$. We conclude that any solution reaches the domain ${\cal D}_0$ possibly in infinite  time.
	
\end{proof}
	 	
\section{The NSN strategy}
\label{secNSN}

Let us fix an initial condition $(x_0,y_0)$ in ${\cal D}_+$ and consider the uncontrolled dynamics, i.e.~with $u=0$.
As long as its solution, denoted $(x^0(\cdot),y^0(\cdot))$,
 belongs to ${\cal D}_+$, $y^0(\cdot)$ is increasing. From Lemma \ref{lemD+}, we know that $y^0(\cdot)$ reaches in finite time the  domain ${\cal D}_-$, where it is decreasing, and finally remains in ${\cal D}_-$. Therefore $y^0(\cdot)$ reaches  its maximum $y^0_{max}<+\infty$ in finite time, and for any $\bar y \in [y_0,y^0_{max}]$, we can define
\begin{equation}
	\label{def_xbar}
\bar x(\bar y) :=x^0(\bar t_{\bar y})\mbox{ where } \bar t_{\bar y}:=\inf\{ t \geq 0; \;y^0(t) =\bar y\} < +\infty
\end{equation}
Note that $x^0(\cdot)$ is decreasing by Assumption \ref{ass2}.i, and therefore the map $\bar y \mapsto \bar x(\bar y)$ is smooth with $\bar x^\prime <0$.

\medskip

We define the NSN (for "Null-Singular-Null") strategy as follows:

\begin{definition}
	For $\bar y \in [y_0,y^0_{max}]$, consider the feedback control
	\begin{equation}
		\label{feedback}
		\psi_{\bar y}(x,y):=\begin{cases}
			k(x):=-\dfrac{f_2(x,\bar y)}{g_2(x,\bar y)} , & \mbox{if } y=\bar y \mbox{ and  } (x,	\bar y)\in {\cal D_+} ,\\
			0 , & \mbox{otherwise.}
		\end{cases}
	\end{equation}
	We denote the $L^1$ norm associated to the NSN control
	\[
	{\mathcal L}(\bar y):= \int_0^{+\infty} u^{\psi_{\bar y}}(t)dt, \quad \bar y \in [y_0,y^0_{max}] ,
	\]
	where $u^{\psi_{\bar y}}(\cdot)$ is the open-loop control generated by the feedback \eqref{feedback}.
\end{definition}

Note that the function $k$ is well defined and takes values in $[0,1]$ by Assumption \ref{ass2}.iv.

\medskip

Let us now define the function
\[
\Delta(x,y):=f_2(x,y)g_1(x,y)-f_1(x,y)g_2(x,y), \quad (x,y) \in {\cal D}
\]
Note that $\Delta$ is negative on ${\cal D}_+$ thanks to Assumption \ref{ass2}.
We consider the following assumption
\begin{assum}
	\label{ass3}
 Under Assumption \ref{ass2}, 
	the function $\frac{f_2}{\Delta}$ is increasing w.r.t.~$y$ in ${\cal D}_+$
\end{assum}

Then, the function ${\cal L}$ can be characterized as follows.

\begin{prop} 
	\label{propcalC}
	Under Assumptions \ref{ass1}, \ref{ass2} and \ref{ass3}, one has
	\begin{equation}
		\label{calL}
		{\cal L}(\bar y)=\int_{x_h(\bar y)}^{\bar x(\bar y)}\dfrac{-f_2(x,\bar y)}{\Delta(x,\bar y)}dx, \quad \bar y \in [y_0,y^0_{max}]
	\end{equation}
(where $x_h(\bar y)$ and $\bar x(\bar y)$ are defined in \eqref{def_xh} and \eqref{def_xbar} respectively).
	Moreover, the map $\bar y \mapsto {\cal L}(\bar y)$ is decreasing.
\end{prop} 

\begin{proof}
	Along the solution $(x(\cdot),y(\cdot))$ generated by the feedback control \eqref{feedback}, let us consider the time function
	\[
	\gamma_2(t):=f_2(x(t),\bar y), \quad t \geq \bar t_{\bar y}
	\]
	where $\bar t_{\bar y}=\inf\{ t \geq 0; \; y(t)=\bar y\}<+\infty$. As long as $(x(t),\bar y) \in {\cal D}_+$, one has
	\[
	\dot x = h(x):=f_1(x,\bar y)+g_1(x,\bar y)k(x)=-\frac{\Delta(x,\bar y)}{g_2(x,\bar y)} 
	\]
	which is negative by Assumption \ref{ass2}, and then
	\[
	\dot\gamma_2 = \frac{\partial f_2(x(t),\bar y)}{\partial x}h(x) < 0
	\]
	(the function $f_2$ being increasing w.r.t.~$x$ in ${\cal D}_+$ from Assumption \ref{ass2}.iii).
	The function $\gamma_2$ is thus decreasing and therefore ${\cal D}_0$ is reached at a time $T \geq \bar t_{\bar y}$ (possibly equal to $+\infty$). One has then 
	\[
	x(T)= \max\{ x \leq \bar x(\bar y); \; f_2(x,\bar y)= 0\} =x_h(\bar y)
	\]
	If $T< +\infty$, then $u(t)=0$ for any $t>T$ because the state cannot reaches again ${\cal D}_+$ by Lemma \ref{lemD+}. Therefore, one has
	\[
	{\mathcal L}(\bar y)=\int_{\bar t_{\bar y}}^T k(x(t))dt
	\]
	Note that the maps $[\bar t_{\bar y},T] \mapsto [x_h(\bar y),\bar x(\bar y)]$ is onto, and one can then write
	\[
	{\mathcal L}(\bar y)= -\int_{x_h(\bar y)}^{\bar x(\bar y)} \frac{k(x)}{h(x)}dx=\int_{x_h(\bar y)}^{\bar x(\bar y)}\dfrac{-f_2(x,\bar y)}{\Delta(x,\bar y)}dx
	\]
	The map $\bar y \mapsto x_h(\bar y)$ is not necessarily differentiable. However, the integrand in the above expression of ${\cal L}$ is null at $x_h(\bar y)$ for any $\bar y$. Therefore, ${\cal L}$ is differentiable with
	\[
	{\cal L}^\prime (\bar y) = \left(-\dfrac{f_2(\bar x,\bar y)}{\Delta(\bar x,\bar y)}\right)\bar x^\prime
	-\int_{x_h}^{\bar x}\dfrac{\partial}{\partial y}\left(\dfrac{f_2(x,\bar y)}{\Delta(x,\bar y)}\right)dx
	\]
	By Assumptions \ref{ass2} and \ref{ass3}, one has $-\frac{f_2}{\Delta}<0$ and $\frac{\partial}{\partial y}(\frac{f_2}{\Delta})>0$ on ${\cal D}_+$, and as $\bar x^\prime$ is negative, we deduce that one has ${\cal L}^\prime (\bar y) <0$.
\end{proof}

\begin{rem}
\label{rem_feedback}
    When applying the feedback \eqref{feedback}, it generates only one discontinuity point of the open loop control $u^{\psi_{\bar y}}(\cdot)$, when the solution $y(\cdot)$ reaches $\bar y$ in ${\cal D}_+$, but not when $y(\cdot)$ leaves the singular arc $y=\bar y$ as the control is null when reaching ${\cal D}_0$. Consequently, the trajectory tangentially leaves the singular arc.
\end{rem}

%%%%%%%%%%%%%%%%%%%%%%%%%%%%%%%%%%%%%%%%%%%%%%%%%%%%%%%%%%%%%%%%%%%%%%

 \section{An optimal synthesis}
 \label{secmain}
 
 In this section, we give our main result about the optimality of the NSN strategy, which is expressed in terms of positivity of a certain flux on the domain ${\cal D}_+$. The proof is using the clock form but in an unusual way (compared for instance to \cite{Miele,HermesLaSalle}), which requires some assumption about the dynamics on the boundary of the domain ${\cal D}_+$, given below.
 \begin{assum}
 	\label{ass4}
 	In ${\cal D}_0$, $g_2$ is negative and one has $\nabla f_2.(f+g) \geq 0$.
 \end{assum}

\begin{prop}
		\label{mainprop}
Under Assumptions \ref{ass1}, \ref{ass2}, \ref{ass3} and \ref{ass4}, let  $(x_0,y_0)$ be an initial condition in  ${\cal D}_+$ such that $\mathcal{L}(y_0)\geq K$.
If one has
	   \begin{equation}
	   	\label{condGreen}
		\frac{\partial}{\partial y}\left(\frac{f_2(x,y)}{\Delta(x,y)}\right) +
		\frac{\partial}{\partial x}\left(\frac{f_1(x,y)}{\Delta(x,y)}\right) >0, \quad (x,y) \in {\cal D}_+, \; y \leq y^0_{max}
		\end{equation}
then the feedback $\psi_{\bar y^*}$, with $\bar y^*\in[y_0,y^0_{max}]$ such that $\mathcal{L}(\bar y^*)=K$, is optimal.
	\end{prop}
	 
\begin{proof}

	 Note first that one has ${\cal L}(y^0_{max})=0$ as the NSN control is identically null for $\bar y=y^0_{max}$. As the map ${\cal L}$ is decreasing by Proposition \ref{propcalC}, we deduce that there exists an unique $\bar y^\star \in [y_0,y_{max}]$ such that ${\cal L}(\bar y^\star)=K$ when ${\cal L}(y_0)\geq K$.
	 
	 \medskip

	 For a fixed initial condition $(x_0,y_0)$ in ${\cal D}_+$, we denote by $(x^\star(\cdot),y^\star(\cdot))$ the solution generated by the NSN strategy with $\bar y = \bar y^\star$, and $u^\star(\cdot)$ its open loop control.
	 Consider the curve ${\cal C}^\star$ in the plane
	 \[
	 {\cal C}^\star:= \left\{ (x^\star(t),y^\star(t)) \; ; \; t \in [0, t_h^\star] \right\}
	 \]
	 where $t_h^\star>0$ is such that $x^\star(t_h^\star)=x_h(\bar y^\star)$. ${\cal C}^\star$ is the part of the orbit for which $y^\star(\cdot)$ is non decreasing, and its extremity  belongs to ${\cal D}_0$.
	 
	   Let $\bar t^\star \in [0,t_h^\star]$ be such that $x^\star(\bar t^\star)=\bar x(\bar y^\star) \leq x_0$.
	 For any $t \in [0,\bar t^\star]$, the control $u^\star(t)$ is null. Then, at any $(x,y) \in {\cal C}^\star$ with $x > \bar x(\bar y^\star)$, the curve ${\cal C}^\star$ admits an upward normal in the $(x,y)$ plane given by
	 \[
	 \vec n(x,y)=\left[\begin{array}{c}
	 f_2(x,y)\\
	 -f_1(x,y)\end{array}\right]
	 \]
	 Let $(x,y) \mapsto \vec v(x,y,u)$ be the vector field in the plane for the control $u$. For any $(x,y) \in {\cal C}^\star$ with $x > \bar x(\bar y^\star)$, one has 
	 	\[
	 \vec n(x,y).\vec v(x,y,u)=\Delta(x,y)u\leq 0 .
	 \]
	 Therefore, the forward orbit with any other control $u(\cdot)$ lies below the curve ${\cal C}^\star$ in the $(x,y)$ plane for $x \in [\bar x(\bar y^\star),x_0]$. 
	 
	 \medskip
	 
	 Assume that there exists another solution $(x(\cdot),y(\cdot))$ with $(x(0),y(0))=(x_0,y_0)$, generated by an optimal control $u(\cdot)$ such that $\sup_t y(t) < \bar y^\star$.
	 From Lemma \ref{lemD+}, we know that $(x(\cdot),y(\cdot))$ reaches the level set ${\cal D}_0$ at a time $t^0$ (possibly infinite). The trajectory being bounded, by Assumption \ref{ass1}.i, the point $(x(t^0),y(t^0)) \in {\cal D}_0$ is finite, with $y(t^0)<\bar y^\star$. Let
	 \[
	 {\cal C}:=\{ (x(t),y(t)), \; ; \; t \in [0, t^0]\}
	 \]
	 that has to be below ${\cal C}^\star$, according to the above.

	From the point $(x^\star(t_h^\star),y^\star(t_h^\star)) \in {\cal D}_0$, there exists an admissible trajectory that stays in the level set ${\cal D}_0$ if for any $(x,y)\in{\cal D}_0$ there is a control $u$ in $[0,1]$ such that
	 \[
	 \nabla f_2(x,y).(f(x,y)+g(x,y)u)=0
	 \]
	 On the set ${\cal D}_0$, one has $\nabla f_2.f=\partial_x f_2 f_1$ which is negative by Assumptions \ref{ass2}.v-vi. Then the function
	 \[
	 	 \psi^\dag(x,y):=-\frac{\nabla f_2(x,y).f(x,y)}{\nabla f_2(x,y).g(x,y)} > 0 , \quad (x,y) \in {\cal D}_0
	 \]
	 is well defined and belongs to $[0,1]$ by Assumption \ref{ass4}.
	 Let $(x^\dag(\cdot),y^\dag(\cdot))$ be the solution of \eqref{sys} with $(x^\dag(t_h^\star),y^\dag(t_h^\star))=(x^\star(t_h^\star),y^\star(t_h^\star))\in {\cal D}_0$ and the feedback control $\psi^\dag$. We denote $u^\dag(\cdot)$ the corresponding open loop control. The trajectory remains in ${\cal D}_0$, and as $g_2$ is negative on ${\cal D}_0$ (Assumption \ref{ass4}) and $u^\dag$ is positive, one has $\dot y^\dag<0$. Therefore, there exists $t^\dag<+\infty$ such that 
	 $y^\dag(t^\dag)=y(t^0)$, with $x^\dag (t^\dag)=x_h(y(t^0))=x(t^0)$. Let
	 \[
	 {\cal C}^\dag := \{ (x^\dag(t),y^\dag(t)) \; ; \; t \in [t_h^\star, t^\dag]\}
	 \]

	 We consider now the concatenation of the three curves ${\cal C}^\star$, ${\cal C}^\dag$ and ${\cal C}$ (see Figure \ref{fig:green}), which defines a simple closed curve $\Gamma= \{ (\tilde x(\tau), \tilde y(\tau)) \; ; \; \tau \in [0,t^\dag+t^0) \}$ with
	 \[
	 (\tilde x(\tau), \tilde y(\tau)) = \begin{cases}
	 	(x^\star(\tau),y^\star(\tau)), & \tau \in [0,t_h^\star)\\
	 	(x^\dag(\tau),y^\dag(\tau)), & \tau \in [t_h^\star,t^\dag)\\
	 	(x(t^0-t^\dag-\tau),y(t^0-t^\dag-\tau)), & \tau \in [t^\dag,t^\dag+t^0)	 	
	 \end{cases}
	 \]
	 that is anti-clockwise oriented in the $(x,y)$ plane by $\tau \in [0,t^\dag+t^0)$ (see Figure \ref{fig:green}). Let ${\cal E}$ be the region bounded by $\Gamma$, which belongs to ${\cal D}_+$. By Assumption \ref{ass2}, $\Delta$ is non null on ${\cal D}_+$ and one can then write from equations \eqref{sys} the 1-form in ${\cal E}$
	 \[
	 u(t)dt=\frac{f_2(x,y)}{\Delta(x,y)} dx - \frac{f_1(x,y)}{\Delta(x,y)} dy
	 \]
	 Applying Green's Theorem, one obtains
	 \[
	 \oint_\Gamma u(t)dt = \iint_{\cal E} \frac{\partial}{\partial x}\left( -\frac{f_1(x,y)}{\Delta(x,y)}\right) -\frac{\partial}{\partial y}\left(\frac{f_2(x,y)}{\Delta(x,y)}\right) dxdy
	 \]
	 which is negative by condition \eqref{condGreen}. Consequently, one has
	 \[
	 \oint_\Gamma u(t)dt = \int_0^{t_h^\star} u^\star(t)dt + \int_{t_h^\star}^{t^\dag} u^\dag(t)dt - \int_0^{t_0}u(t)dt < 0
	 \]
	 that is
	 \[
	 \int_0^{t_0}u(t)dt > K+ \int_{t_h^\star}^{t^\dag} u^\dag(t)dt > K
	 \]
	 which contradicts the optimality of the control $u(\cdot)$ under the constraint \eqref{constraint}. 
	 \end{proof}
 
 \begin{figure}[h!t]
 	\centering
 	\includegraphics[width=0.9\textwidth]{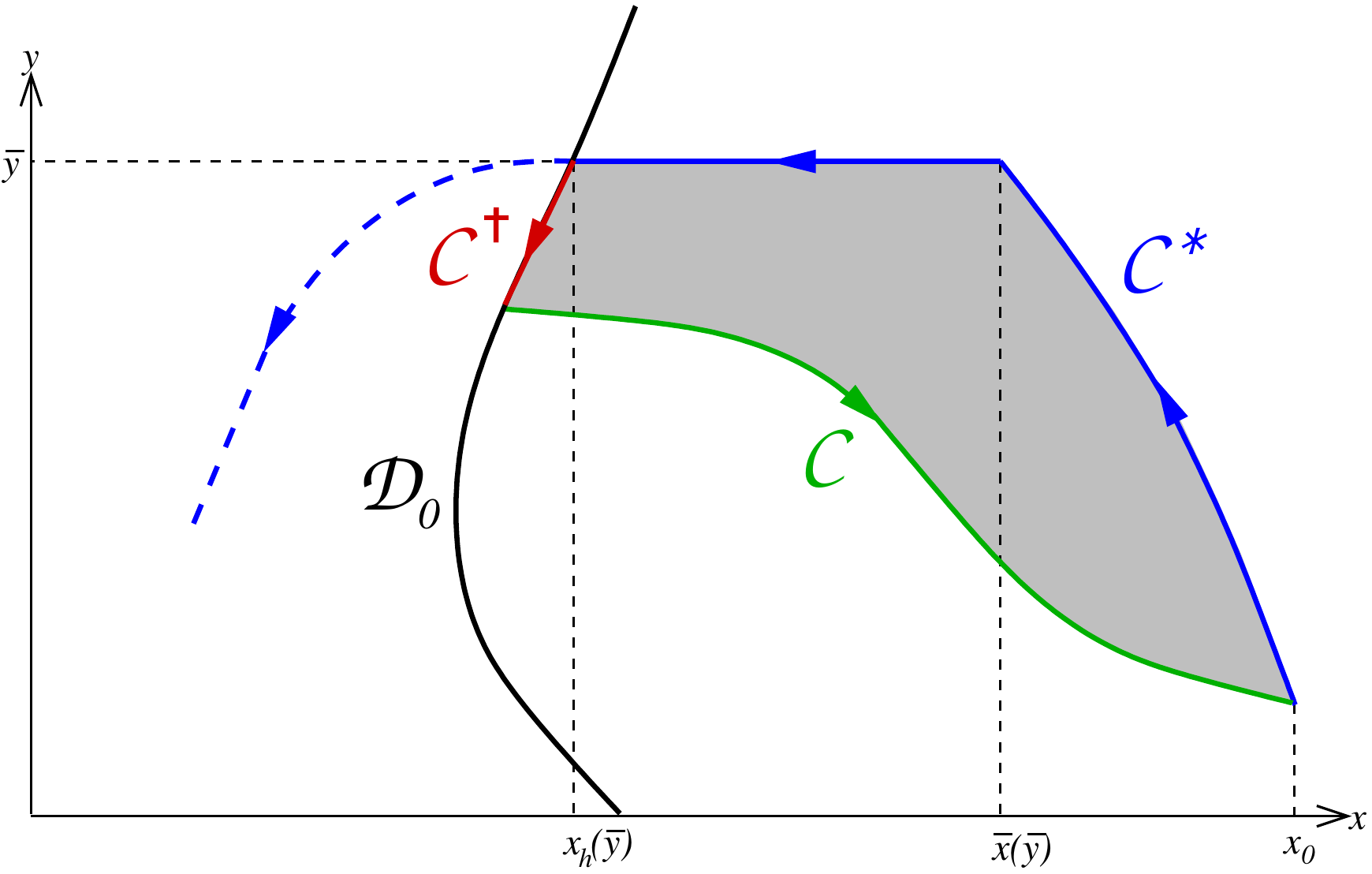}
 	\caption{Application of the Green's Theorem on the closed domain (in gray) delimited by the curves ${\cal C}^\star$ (in blue), ${\cal C}^\dag$ (in red) and ${\cal C}$ (in green). The curve (in black) represents the level-set ${\cal D}_0$ which delimits the domain ${\cal D}_-$ on the right.}
 	\label{fig:green}
 \end{figure}
 
 \begin{rem}
 	When $K > \mathcal{L}(y_0)$, the budget $K$ is large enough to ensure $y(t)\leq y_0$ for any $t \geq 0$. One can apply for instance the feedback strategy $\psi_{y_0}$, which is optimal with a $L^1$ norm of the control less than $K$, equal to ${\cal L}(y_0)$.
 \end{rem}

	 \bigskip
	 
	 Let us illustrate our results on an example for which the optimal control can be determined analytically.
	 
\begin{example}
	\label{ex-aca}
	\[
	\left\{\begin{array}{lll}
	\dot x & = & -(x+1)^2y + (x+1)^2y u\\
	\dot y & = & xy - (x+1)yu
	\end{array}
	\right. \qquad u \in [0,1]
	\]
	Whatever is the control $u$, one has $\dot x=0$ at $x=-1$ and $\dot y=0$ at $y=0$. Therefore the domain 	
	 \[
	 	{\cal D}= \{ (x,y) \in \Rset^2 ; \; x>-1, \; y>0 \}
	 \]
  is invariant. Let $V(x,y)=x+2y$. With $u=0$, one has $\frac{d}{dt}V=-x^2y-y^2<0$. The function $V$ is thus decreasing along the solutions in ${\cal D}$, from which one deduces the inequalities
  $-1 \leq x(t) \leq x(0)+2y(0)$ and $ 0 \leq y(t) \leq (x(0)+y(0)-1)/2$. The solutions in ${\cal D}$ with $u=0$ are thus bounded.

  The sub and super sets of the function $f_2$ are ${\cal D}_-=\{ (x,y) \in {\cal D} ; \; x <0\}$, ${\cal D}_+=\{ (x,y) \in {\cal D} ; \; x >0\}=\{ (x,y) \in \Rset^2 ; \; x>0, \; y>0\}$ and the function $x_h$ is simply the null function. Assumption \ref{ass1} is satisfied.

  \medskip

  In ${\cal D}_+$, the function $f_1(x,y)=-(x+1)^2y$ is negative and decreasing with respect to $x$ and $y$, while the function $g_1(x,y)=-f_1(x,y)$ is increasing with respect to $x$ and $y$ with $f_1+g_1=0 \leq 0$. The function $f_2(x,y)=xy$ is increasing with respect to $x$, while the function $g_2(x,y)=-(x+1)y$ is decreasing with respect to $x$ and $y$ with $f_2+g_2=-y<0$. In ${\cal D}_0$, one has $f_1=-y<0$ and $\partial_x f_2=y>0$, $\partial_y f_2=0 \leq 0$. Assumption \ref{ass2} is thus fulfilled.

  \medskip

  One has
  \[
\frac{f_2}{\Delta}=-\frac{x}{(x+1)^2y}
  \]
  which is increasing with respect to $y$ on ${\cal D}_+$. Assumption \ref{ass3} is satisfied.

  \medskip

  In ${\cal D}_0$, $g_2=-y$ is negative and $\nabla f_2.(f+g)=0 \geq 0$. Assumption \ref{ass4} is satisfied.

  \medskip
  Finally, one has
  \[
\partial_y\left(\frac{f_2}{\Delta}\right) + \partial_x \left(\frac{f_1}{\Delta}\right)
		= \frac{x}{(x+1)^2y}
  \]
  that is positive on ${\cal D}_+$. 
  
  Now, from Proposition \ref{propcalC}, one can determine the function ${\cal L}$ as follows.
  Firstly, the solution of the system with $u=0$ for an initial condition $(x_0,y_0)$ in ${\cal D}_+$ can be parameterized by $x$ as the map $t \mapsto x(t)$ is decreasing, that is
  \[
y(t)= y_0 - \int_{x_0}^{x(t)} \frac{x}{(x+1)^2} \, dx
  \]
which gives 
\begin{equation}
\label{ybar-ex1}
\int_{x_0}^{\bar x(\bar y)} \frac{x}{(x+1)^2} \, dx= y_0 - \bar y
\end{equation}
and
\begin{equation}
\label{L-ex1}\
\begin{array}{lllll}
y_{max}^0  & =  & \ds y_0-\int_{x_0}^{0} \frac{x}{(x+1)^2}dx  & = & \ds y_0-\left[\frac{1}{x+1}+\log(x+1)\right]_0^{x_0}\\[4mm]
  & & & = & \ds \frac{1}{x_0+1}+\log(x_0+1)-1+y_0
\end{array}
\end{equation}
Secondly, one has
 \[
 \begin{array}{lllll}
{\cal L}(\bar y) & = & \ds \int_0^{\bar x(\bar y)} \frac{-f_2(x,\bar y)}{\Delta(x,\bar y)} \, dx 
& = & 
\ds \frac{1}{\bar y}\int_0^{\bar x(\bar y)}  \frac{x}{(x+1)^2} \, dx\\[4mm]
& & & = & \ds \frac{1}{\bar y}\left(\int_0^{x_0}\frac{x}{(x+1)^2} \, dx + \int_{x_0}^{\bar x(\bar y)}\frac{x}{(x+1)^2} \, dx\right)
\end{array}
\]
which gives with \eqref{ybar-ex1} and \eqref{L-ex1} the expression
\[
{\cal L}(\bar y)=\frac{y_{max}^0-y_0+y_0-\bar y}{\bar y}=\frac{y_{max}^0}{\bar y}-1
 \]
that is defined for $\bar y \in [y_0,y_{max}^0]$. Finally, from Proposition \ref{mainprop}, we obtain that for a budget $K \leq \frac{y_{max}^0}{y_0}-1$, the NSN strategy \eqref{feedback} with
\[
\bar y=\bar y^\star := \frac{y_{max}^0}{K+1}
\]
is optimal. Therefore, the feedback
\[
\psi^\star(x,y):= \begin{cases}
			\frac{x}{x+1}, & \mbox{if } y=\bar y^\star \mbox{ and  } x>0\\
			0 , & \mbox{otherwise.}
		\end{cases}
\]
is optimal. An example of optimal solution is drawn on Figure \ref{fig:ex_aca}, where one can see that the optimal trajectory leaves tangentially the singular arc and the optimal control is continuous at that point, as underlined in Remark \ref{rem_feedback}.

\begin{figure}[h!t]
 	\centering
 	\includegraphics[width=0.45\textwidth]{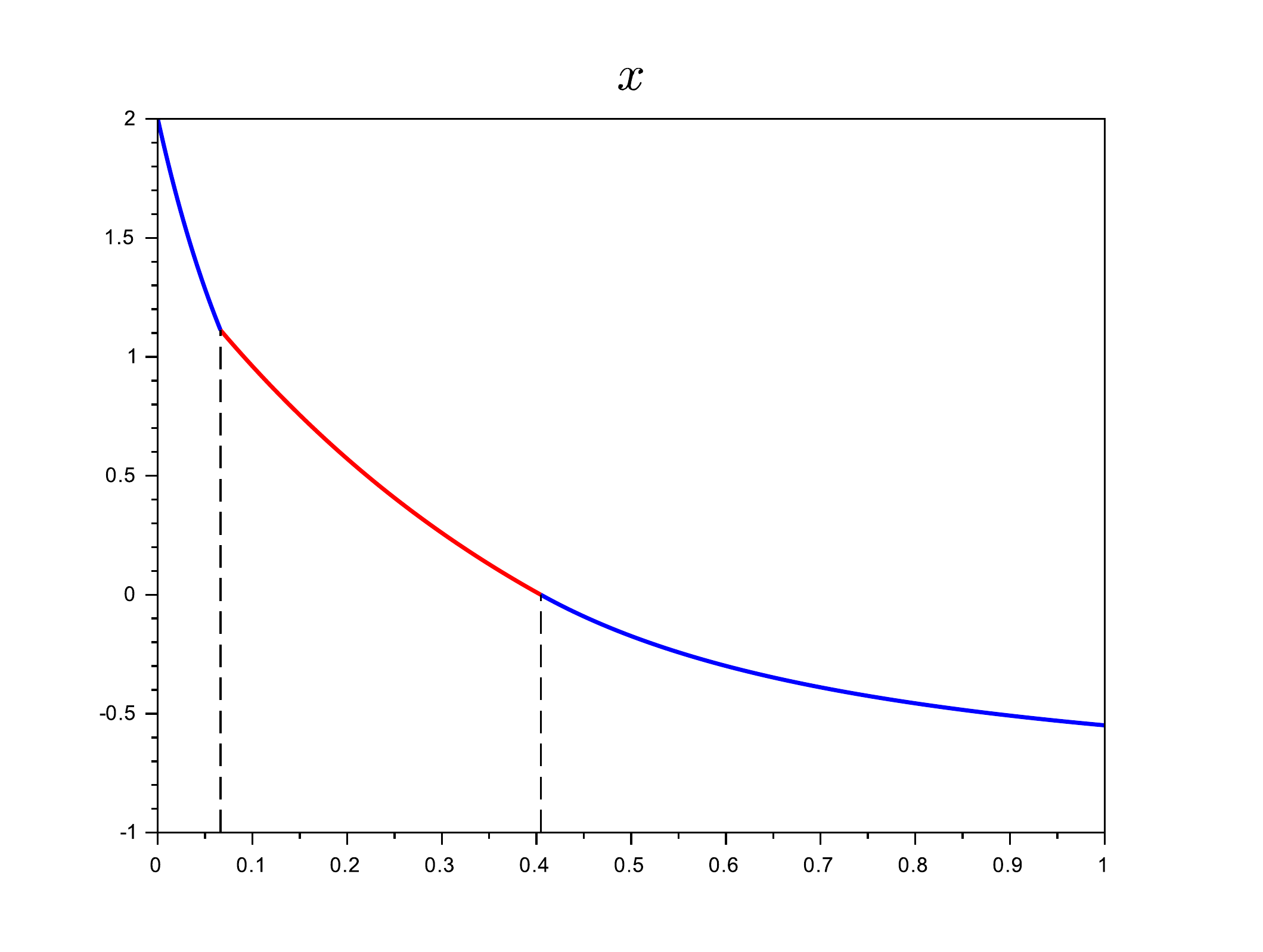}
  \includegraphics[width=0.45\textwidth]{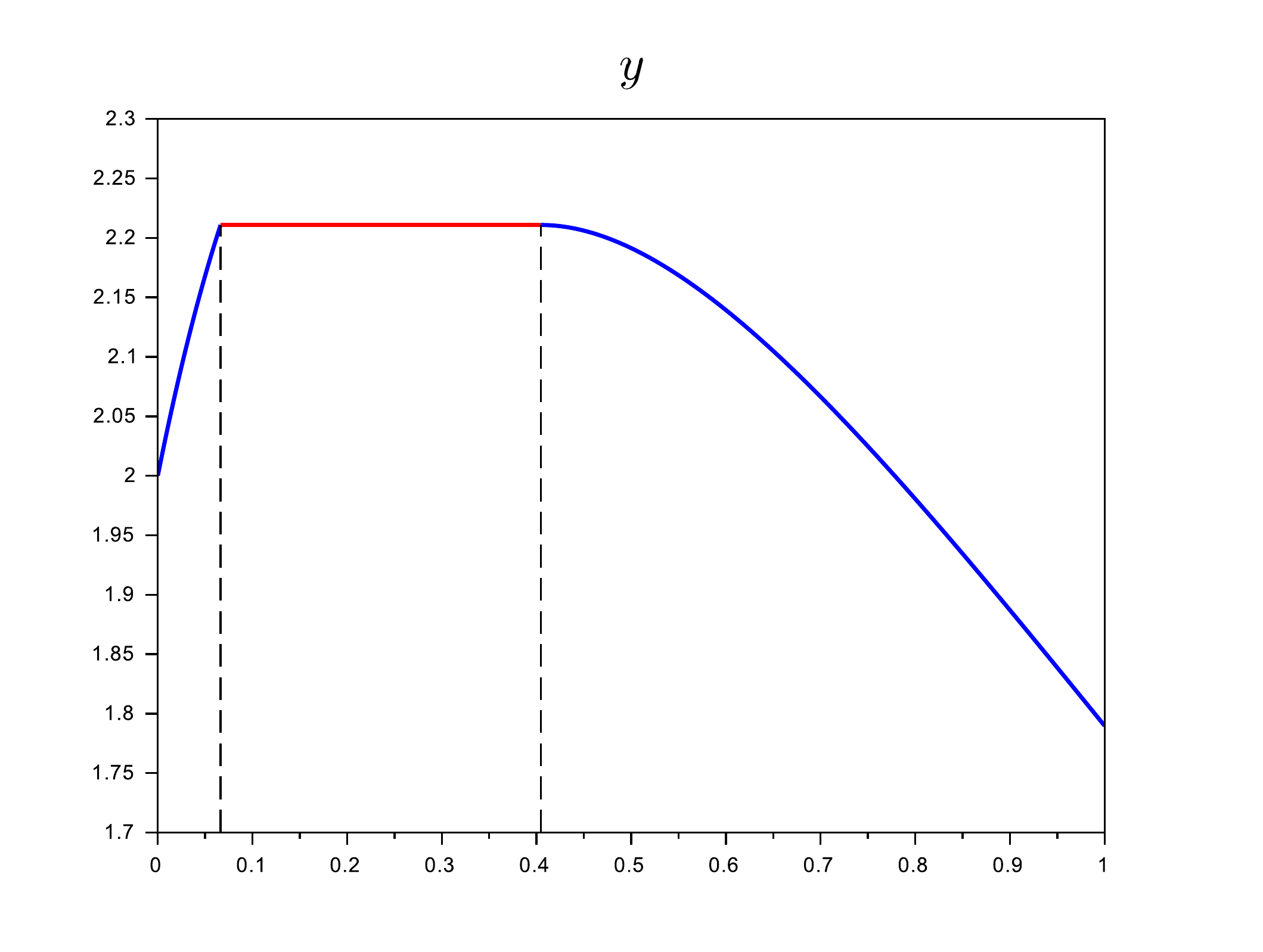}
  \includegraphics[width=0.45\textwidth]{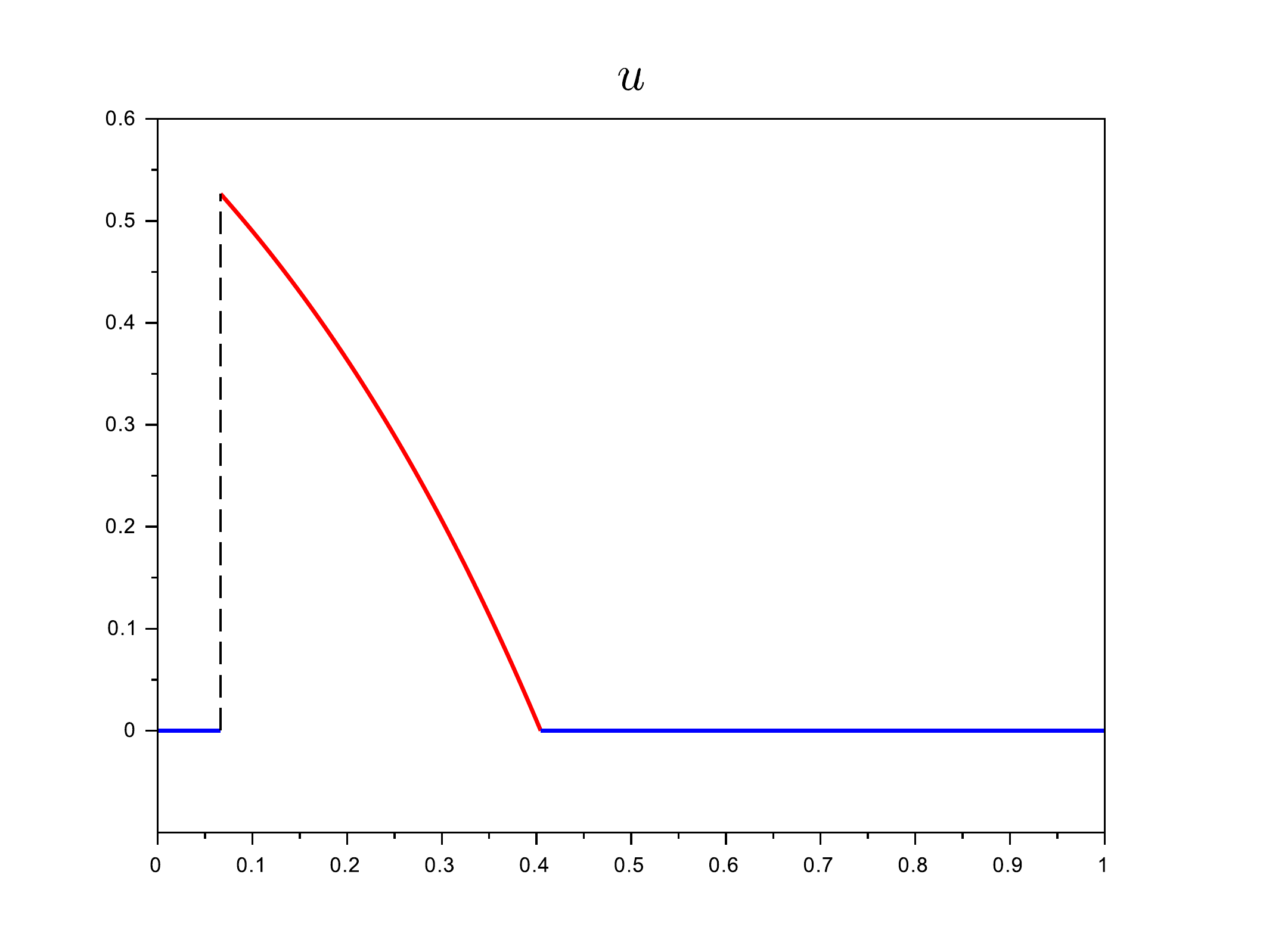}
 	\caption{Optimal solution of Example \ref{ex-aca} for initial condition $(x_0,y_0)=(2,2)$ and budget $K=0.1$ (singular arc is depicted in red).}
 	\label{fig:ex_aca}
 \end{figure}

  \end{example}

%%%%%%%%%%%%%%%%%%%%%%%%%%%%%%%%%%%%%%%%%%%%%%%%%%%%%%%%%%%%%%
	 
	 \section{The case of Kolmogorov dynamics}
	 \label{secparticular}
	 	In this Section we particularize the results of Proposition \ref{mainprop} to a class of Kolmogorov dynamics in $\Rset_+^2$, for which it is easier to verify the required assumptions.
	 \begin{equation}
	 \label{sys2}
	 \left\{\begin{array}{lll}
	 \dot x & = & -\big(\phi_1(x,y)-\phi_2(x,y)u\big)x\\
	 \dot y & = & \big(\phi_3(x,y)-\phi_4(x,y)u\big)y
	 \end{array}
	 \right. \qquad u \in [0,1]
	 \end{equation}
	 where $\phi_i$ are smooth maps. The positive orthant ${\cal D}=\{ (x,y)\in\Rset^2; \; x >0, \; y >0\}$ is clearly invariant by \eqref{sys2}.  
	 
	 \begin{hypos}
	 	\label{hypKolmogorov}
	 	On ${\cal D}$, one has
	 \begin{enumerate}[label=\roman*.]
	 	\item $\phi_1$ and $\phi_2$ are positive, with $\phi_2-\phi_1 \leq M < +\infty$.
	 	\item $\phi_3$ is increasing with respect to $x$ and non increasing with respect to $y$, with $\phi_3(0,y) < 0 < \lim_{x \to+\infty}\phi_3(x,y)$ for any $y>0$.
	 \item $\phi_1\geq \phi_2$ with $\partial_y\phi_1 \geq \partial_y\phi_2>0$ when $\phi_3\geq 0$, and $\phi_1=\phi_2$ when $\phi_3=0$. The maps $x \mapsto \phi_1(x,y)x$, $x \mapsto \phi_2(x,y)x$ are increasing for any $y$.
	 	\item When $\phi_3>0$, $\phi_4$ is increasing with respect to $x$ with
	 	$\phi_4 > \phi_3$ and  $[\phi_3,\phi_4]_y:=\phi_3\partial_y \phi_4-\phi_4\partial_y\phi_3\geq 0$. The map $y \mapsto \phi_4(x,y)y$ is increasing for any $x$.
	 \end{enumerate}
 \end{hypos}

\begin{lem}
	 Under Hypotheses \ref{hypKolmogorov}, Assumptions \ref{ass1}, \ref{ass2}, \ref{ass3}, \ref{ass4} are fulfilled.
\end{lem}

\begin{proof}
 
 From hypothesis \ref{hypKolmogorov}.ii, there exists a unique  map $y \mapsto x_h(y)>0$ such that $\phi_3(x_h(y),y)=0$ for any $y>0$, which is moreover non decreasing with respect to $y$. The sub and super level sets ${\cal D}_-$, ${\cal D}_+$, are thus non empty and defined as
 \[
 {\cal D}_-=\{ (x,y)\in{\cal D}; \; x < x_h(y)\}, \quad
 {\cal D}_+=\{ (x,y)\in{\cal D}; \; x > x_h(y)\}
 \]
 and the level set ${\cal D}_0$ is $\{ (x,y)\in{\cal D}; \; x = x_h(y)\}$.
 
 From Hypothesis \ref{hypKolmogorov}.i, one has $\dot x\leq \big(\phi_1(x,y)(u-1)+Mu\big)x$ from which one  can write
 for any admissible solution
 \[
 x(t) \leq x_0e^{MK} \exp\left(\int_0^t \phi_1(x(\tau),y(\tau))(u(\tau)-1) \; d\tau\right) \leq x_0e^{MK} < +\infty
 \]
For the uncontrolled dynamics, one has $\dot x = -\phi_1(x,y) x < 0$ i.e.~$x(\cdot)$ is decreasing. Let us show that ${\cal D}\setminus {\cal D_+}={\cal D}_0\cup{\cal D}_-$ is reached in finite time. If not, one has $x(t)\geq x_h(y(t))$ for any $t\geq 0$ and $y(\cdot)$ is increasing. Then,  one should have  $\dot x(t) \leq -\kappa x(t))$ for any $t\geq 0$, where $\kappa=\min_{\xi \in [x_h(y_0),x(0)]}\phi_1(\xi,y(0))>0$.
 Therefore, $x(\cdot)$ converges to $0$, while $x_h(y(t))\geq x_h(y(0))>0$ for any $t \geq 0$, and thus a contradiction.  At ${\cal D}_0$, one has
 \[
 \frac{d}{dt}\phi_3(x,y)= -\partial_x\phi_3(x,y)\phi_1(x,y)x < 0
 \]
 with Hypotheses \ref{hypKolmogorov}i.~and ii. The domain ${\cal D}\setminus {\cal D_+}$ is thus (positively) invariant.  Moreover one has $\dot y \leq $ in ${\cal D}\setminus {\cal D_+}$. We conclude that the solutions for the uncontrolled dynamics are either non decreasing, or increasing up to a finite time and then non decreasing (and thus bounded). Moreover, any other controlled solution with a lower peak of $y$ is also bounded, as $x(\cdot)$ is always bounded. Assumption \ref{ass1} is verified.
 
 \medskip

 Clearly, the map $f_1=-\phi_1x$ is negative in ${\cal D}_+ \cup {\cal D}_0$ and decreasing with respect to $x$ and $y$ from Hypotheses \ref{hypKolmogorov}.i. and iii. The map $g_1=\phi_2 x$ is increasing with respect to $x$ and $y$ also from Hypothesis \ref{hypKolmogorov}.iii, and $f_1+g_1=(\phi_2-\phi_1)x\leq 0$  in  ${\cal D}_+$. The map $f_2=\phi_3 y$ is increasing with respect to $x$ from Hypothesis \ref{hypKolmogorov}.ii, and $\partial_y f_2=\phi_3 + \partial_y \phi_3 y$ is non positive on ${\cal D}_0$. The map $g_2=-\phi_4 y$ is decreasing with respect to $x$ and $y$ from Hypothesis \ref{hypKolmogorov}.iv, and $f_2+g_2=(\phi_3-\phi_4)y <0$ in  ${\cal D}_+$. Assumption \ref{ass2} is verified.

 \medskip

 One has $\Delta=(\phi_3\phi_2-\phi_1\phi_4)xy$ and with inequality $\phi_1\geq \phi_2$, one  obtains  $\Delta \leq (\phi_3-\phi_4)\phi_1xy$ on ${\cal D}_+$, which is negative by Hypothesis \ref{hypKolmogorov}.iv. One gets
 \[
 \dfrac{f_2}{\Delta}=\dfrac{\phi_3}{(\phi_3\phi_2-\phi_1\phi_4)x} 
 \Rightarrow 
 \dfrac{\partial}{\partial y} \left( \dfrac{f_2}{\Delta} \right) =  \dfrac{\phi_3(\phi_4\partial_y\phi_1-\phi_3\partial_y\phi_2)+\phi_1[\phi_3,\phi_4]_y}{(\phi_3\phi_2-\phi_1\phi_4)^2x}
 \]
 With Hypotheses \ref{hypKolmogorov}.iii. and iv., one has $[\phi_3,\phi_4]_y\geq 0$ and $\phi_4\partial_y\phi_1-\phi_3\partial_y\phi_2>0$ on ${\cal D}_+$, which implies that $\partial_y f_2/\Delta$ is positive on ${\cal D}_+$.
 Assumption \ref{ass3} is thus verified.
 
 \medskip

  From Hypothesis \ref{hypKolmogorov}.iv, one has $\phi_4>0$ on ${\cal D}_0$. Then, 
  the map $g_2=-\phi_4y$ is negative on ${\cal D}_0$. Moreover, one has
 \[
 \nabla f_2=\left[\begin{array}{c}
 	\partial_x\phi_3\\
 	\partial_y\phi_3
 	\end{array}\right]y, \quad (x,y) \in {\cal D}_0
 \]
 which gives
 \[
 \nabla f_2.(f+g)  =  \partial_x\phi_3 (\phi_2-\phi_1)xy +  \partial_y\phi_3(\phi_3-\phi_4)y^2
 \]
 that is non negative on ${\cal D}_0$ with Hypotheses \ref{hypKolmogorov}.ii, iii and iv. Assumption 4 is fulfilled.
 
 \end{proof}
 
 Let us  posit
 \[
 \delta(x,y)=\phi_3(x,y)\phi_2(x,y)-\phi_1(x,y)\phi_4(x,y), \quad (x,y) \in {\cal D}
 \]
The application of Propositions \ref{propcalC} and \ref{mainprop} gives the following result.

 \begin{prop}
 	\label{prop2}
 	Under Hypotheses \ref{hypKolmogorov}, one has
 	\[
 	{\cal L}(\bar y)=\int_{x_h(\bar y)}^{\bar x(\bar y)}\dfrac{-\phi_3(x,\bar y)}{\delta(x,\bar y)x} \, dx
 	\]
 	For initial conditions $(x_0,y_0)$ in  ${\cal D}_+$ such that $\mathcal{L}(y_0)\geq K$ and
 		\begin{equation}
 	\label{condGreenbis}
 	\begin{array}{l}
 	(\phi_3(\phi_4\partial_y\phi_1-\phi_3\partial_y\phi_2)+\phi_1[\phi_3,\phi_4]_y)y \; + \; \\
 	 \qquad \quad (\phi_3[\phi_1,\phi_2]_x+\phi_1[\phi_2\partial_x\phi_3-\phi_1\partial_x\phi_4])x > 0, \; 
 	 (x,y) \in {\cal D}_+, \; y  \leq y^0_{max} 
 	 \end{array}
 	\end{equation}
 	(where $[\phi_1,\phi_2]_x:=\phi_1\partial_x\phi_2-\phi_2\partial_x\phi_1$), 
 	then there exists $y^*\in[y_0,y^0_{max}]$ such that $\mathcal{L}(y^*)=K$ and the feedback
 	\begin{equation}
 	\label{feedbackKolmogorov}
 	\psi_{y^\star}(x)= \begin{cases}
 		\dfrac{\phi_3(x,y^\star)}{\phi_4(x,y^\star)} , & \mbox{if } y=y^\star \mbox{ and  } x > x_h(y^\star)\\
 		0 , & \mbox{otherwise.}
 	\end{cases}
 	\end{equation}
 	is optimal.
 \end{prop}

\begin{proof}
	One has
	\[
		\dfrac{f_1}{\Delta}=\dfrac{-\phi_1}{(\phi_3\phi_2-\phi_1\phi_4)y} 
		\Rightarrow 
		\dfrac{\partial}{\partial x} \left( \dfrac{f_1(x,y)}{\Delta(x,y)} \right) = 
	 \dfrac{\phi_3[\phi_1,\phi_2]_x+\phi_1[\phi_2\partial_x\phi_3-\phi_1\partial_x\phi_4]}{(\phi_3\phi_2-\phi_1\phi_4)^2y}
	\]
	and then condition \eqref{condGreen} amounts exactly to have \eqref{condGreenbis}.
\end{proof}

\bigskip

Let us underline that the first term in \eqref{condGreenbis} is necessarily positive, under Hypotheses \ref{hypKolmogorov}.i, iii and iv. A simple way to guarantee condition \eqref{condGreenbis} to be fulfilled is to have the second term non-negative, which can be obtained for instance as follows. 

\begin{coro}
	\label{coro}
		Under Hypotheses \ref{hypKolmogorov} with $\phi_1=\phi_2$ and $\phi_4=\phi_3+\alpha$ ($\alpha>0$) in ${\cal D}$, the feedback \eqref{feedback} is optimal for any initial condition in ${\cal D}_+$ with  $\mathcal{L}(y_0)\geq K$.
\end{coro}

\medskip

We present below some concrete examples within the biological field, which satisfy the conditions of Corollary \ref{coro} and allow to conclude directly about the optimality of the NSN strategy.

\begin{example}
	\label{ex1}
	We consider the SIR model \cite{KermackMcKendrick} which is very popular in epidemiology. With non-pharmaceutic interventions which consist in reducing the contact between susceptible and infected populations (by means of reducing social distance for human disease for instance), the model writes:
	\[
	\left\{\begin{array}{lll}
	\dot x & = & -\beta(1-u) xy\\
	\dot y & = & \beta(1-u) xy -\alpha y
	\end{array}
	\right. \qquad u \in [0,1]
	\]
	where $x$ and $y$ stand for the density of susceptible and infected populations, respectively, and $u$ is the control variable (naturally subject to a budget constraint). Parameter $\beta$ is the infection rate (without intervention), and $\alpha$ is the recovery rate.
	Without control (i.e. $u=0$), it is well known that the condition for an epidemics outbreak is given by the reproduction number
	\[
	{\cal R}_0:=\frac{\alpha}{\beta}
	\]
	that has to be larger than one. Then, the size of the infected population $y(\cdot)$ increases up to a peak value that could be very high. The objective of the control is to reduce this peak value.
	 Here the domain ${\cal D}$ is
	 \[
	 {\cal D} = \{ (x,y) \in \Rset_+^2 ; \; x+y \leq 1 \}
	 \]
	 and one has the following expressions of the functions $\phi_i$ $(i=1\cdots 4)$.
	\[
	\phi_1(x,y)=\phi_2(x,y)=\beta y, \quad \phi_3(x,y)=\beta x-\alpha, \quad \phi_4(x,y)=\beta x
	\]
	The separatrix  ${\cal D}_0$ between ${\cal D}_-$ and ${\cal D}_+$ is a vertical segment
		\[
	{\cal D}_0 = \left\{ (x,y) \in {\cal D}\; ; \; x= \frac{1}{{\cal R}_0} \right\}
	\]
	%{\cal D}_0 = \left\{ \frac{1}{{\cal R}_0} \right\} \times \left( 0,1-\frac{1}{{\cal R}_0}  \right]
	One can straightforwardly check that Hypotheses \ref{hypKolmogorov} and conditions of Corollary \ref{coro} are fulfilled when ${\cal R}_0>1$. We can then conclude that the NSN strategy is optimal under a $L^1$ budget control on the control $u(\cdot)$, as in \cite{Automatica}. Let us underline that when the initial density $y_0$  of the infected population is very low (which is often the case in face to a new epidemics), the time to reach the minimum peak can be very large, justifying the consideration of an unbounded time horizon. In \cite{Automatica}, it is shown that the optimal control can be determined analytically for the limiting case of of an arbitrary small $y_0$ with an initial density of the susceptible population equal to $1-y_0$.
\end{example}

\begin{example}
	\label{ex2}
	We consider the classical resource-consumer (or "batch" bio-process) model, which is very popular in microbiology (see e.g.~\cite{HLRS17})
	\[
	\left\{\begin{array}{lll}
	\dot x & = & \ds -\frac{1}{Y} \mu(x)y(1-u)\\[3mm]
	\dot y & = & \ds  \mu(x)y(1-u)-m y
	\end{array}
	\right. \qquad u \in [0,1]
	\]
	where $x$ and $y$ are the concentrations of the resource and the consumer, respectively. The function $\mu$ is the specific growth rate, that is assumed to follow the well-known Monod's expression
	\[
	\mu(x):=\frac{x}{K+x}
	\]	
	The parameter $Y$ is the yield coefficient of the transformation of the resource into consumer growth, while
	the parameter $m>0$ is the mortality rate of the consumer (supposed to be relatively low compared to the growth term). Here the control $u$ is an isolation factor (by biological or physical means) which limits the access to the resource for the consumer. When the consumer is a living species that proliferates on the resource in an undesirable way (e.g. bacteria presenting some health risks), an objective is to reduce its peak value for a given budget on the control.
	For this model, the domain ${\cal D}$ is
	\[
	{\cal D}= \{ (x,y) \in \Rset_+^2 ; \; x>0, \; y>0 \}
	\]
	with the functions
	\[
	\phi_1(x,y)=\phi_2(x,y)=\frac{1}{Y} \frac{y}{1+x}, \quad \phi_3(x,y)= \frac{x}{1+x}-m, \quad \phi_4(x,y)= \frac{x}{1+x}
	\] 
	for which can easily check that Hypotheses \ref{hypKolmogorov} and conditions of Corollary \ref{coro} are fulfilled for a mortality rate $m<1$.  Here also, the level set ${\cal D}_0$ which splits the domain ${\cal D}$ into between ${\cal D}_-$ and ${\cal D}_+$ is a vertical line
	\[
	{\cal D}_0 = \left\{ (x,y) \in {\cal D} \; ; \; x= \frac{m}{1-m} \right\}
\]
	Then, we can conclude that the NSN strategy is also optimal for this problem.
\end{example}

\begin{example}
	\label{ex3}
	We consider here the same resource-consumer model as in Example \ref{ex2} but with a ratio-dependent growth rate (see e.g.~\cite{HLRS17})
	\[
	\left\{\begin{array}{lll}
		\dot x & = & \ds -\frac{1}{Y} \mu(x,y)y(1-u)\\[3mm]
		\dot y & = & \ds  \mu(x,y)y(1-u)-m y
	\end{array}
	\right. \qquad u \in [0,1]
	\]
	where $\mu$ is the Contois function
	\[
	\mu(x,y)=\frac{x}{x+y}
	\]
	This model aims to take into consideration a crowding effect when the population of consumers is high, or equivalently that the growth is driven by the ratio "resource by consumer" $x/y$ rather than simply the level of the resource $x$.
	Here also, one can easily check that the corresponding functions
	\[
	\phi_1(x,y)=\phi_2(x,y)= \frac{1}{Y} \frac{y}{x+y}, \quad \phi_3(x,y)=\frac{x}{x+y}-m, \quad \phi_4(x,y)= \frac{x}{x+y}
	\]
	satisfy Hypotheses \ref{hypKolmogorov} and conditions of Corollary \ref{coro} for $m<1$. Let us underline that the function $\phi_3$ depends on both variables, differently to Examples \ref{ex1} and \ref{ex2}, and consequently the function $x_h(\cdot)$ is not constant here. The NSN strategy is again optimal for $m<1$ and the level set ${\cal D}_0$, which gives to the end of the singular arc, is no longer a vertical line:
		\[
	{\cal D}_0 = \left\{ (x,y) \in {\cal D} \; ; \; x= \frac{m}{1-m}y \right\}
	\]
\end{example}

\section*{Acknowledgment}
This work has been partially supported by MIAI@Grenoble Alpes (ANR19-P3IA-0003).

%%%%%%%%%%%%%%%%%%%%%%%%%%%%%%%%%%%%%%%%%%%%%%%%%%%%%%%%%%%%%%%

\end{document}